\documentclass[12pt]{amsart} 

\usepackage{amssymb,amsmath,amsthm}
\usepackage{amscd}
\usepackage{latexsym}
\usepackage{tikz}

\input xypic
\xyoption {all}

\usepackage{color}

\newtheorem{thm}{Theorem}[section]

\newtheorem{cor}[thm]{Corollary}

\newtheorem{defn}[thm]{Definition}
\newtheorem{rem}[thm]{Remark}

\numberwithin{equation}{section}

\newcommand{\al}{\alpha}
\newcommand{\be}{\beta}

\newcommand{\de}{\delta}
\newcommand{\e}{\varepsilon}

\newcommand{\la}{\lambda}

\newcommand{\Om}{\Omega}
\newcommand{\La}{\Lambda}

\newcommand{\R}{\mathbb{R}}
\newcommand{\N}{\mathbb{N}}
\newcommand{\ACq}{AC}
\newcommand{\AC}{AC^{(n)}(\Omega,\mathbb{R}^l)}
\newcommand{\ACH}{AC_H^{(n)}(\Omega,\mathbb{R}^l)}
\newcommand{\LL}{\mathcal{L}}
\newcommand{\po}{\text{-}}

\providecommand{\osc}{\mathop{\rm osc}\nolimits}

\newcommand{\lb}{\label}

\newcommand{\lra}{\longrightarrow}

\newcommand{\DEF}{\buildrel {\mbox{\tiny def}}\over =}

\begin{document}

\title{On the {B}ongiorno's
  notion of absolute continuity}



\author{Beata~Randrianantoanina}

\address{Department of Mathematics\\
Miami University\\
Oxford, OH 45056, USA}

\email{\texttt{randrib@miamioh.edu}}

\author{Huaqiang~Xu}

\address{Department of Mathematics\\
Miami University\\
Oxford, OH 45056, USA}

\email{\texttt{xuhuaqiang1990@gmail.com}}




\begin{abstract}
We  show that the  classes of  $\alpha$-absolutely continuous functions in the  sense of Bongiorno coincide for all $0<\alpha<1$. 
\end{abstract}

\maketitle

\section{Introduction}
\label{intro}

The classical Vitali's definition says that when $\Om\subseteq \R$, a function $f:\Om\lra \R$  is {\it absolutely continuous}  if for all $\e>0$, there exists $\de>0$ so that for every finite collection of disjoint intervals $\{[a_i,b_i]\}_{i=1}^k \subset \Om$ we  have (where $\LL^n$  denotes the Lebesgue measure on $\R^n$)
\begin{equation}\lb{hypp}
\sum_{i=1}^k \LL^1[a_i,b_i]<\de  \Rightarrow \sum_{i=1}^{k} |f(a_i)-f(b_i)|<\e.
\end{equation}

The study of the space of absolutely continuous functions on $[0,1]$ and their generalizations to domains in $\R^n$ is connected to the problem of finding regular subclasses of Sobolev spaces which goes back to Cesari and Calder{\'o}n \cite{C41,Cal51}. 
There are several natural ways of generalizing the definition of  absolute continuity for functions of several variables (cf. \cite{CA1933,Hobson,RR,AAK11,Maly1999,Hencl2002}).

Recently Bongiorno \cite{Bongiorno2005} introduced a generalization of absolute continuity, which is simultaneously similar to Arzel\`a's notion of bounded variation for functions on $\R^2$, cf. \cite{CA1933}, and to Mal\'y's notion of absolute continuity \cite{Maly1999}.

\begin{defn} 
\label{alphaac}(Bongiorno \cite{Bongiorno2005})
Let $0<\al<1$. A function $f:\Om\to\R^l$, where $\Om\subset\R^n$ is open, is said to be
\textit{$\al$-absolutely continuous} (denoted $\al\po\AC$ or   $\al\po\ACq$) if for all $ \varepsilon>0$, 
 there exists $\de>0$, such that for any finite collection of disjoint $\al$-regular intervals 
$\{ [\mathbf{a}_i,\mathbf{b}_i]\subset\Om \}_{i=1}^k$ 
we have 
\begin{equation}\lb{alac}
\sum_{i=1}^k \LL^n ([\mathbf{a}_i,\mathbf{b}_i])<\de \Rightarrow\sum_{i=1}^k|f(\mathbf{a}_i)-f(\mathbf{b}_i)|^n<\e.
\end{equation}
\end{defn}
Here, for $\mathbf{a}\in\R^l$, $|\mathbf{a}|$ denotes the Euclidean norm of $\mathbf{a}$, and we say that  an interval $[\mathbf{a},\mathbf{b}]\DEF \{\mathbf{x}=(x_\nu)_{\nu=1}^n\in\R^n\text{: }a_\nu\leq x_\nu\leq b_\nu,\nu=1,\dots,n\}$ is {\it $\al$-regular} if
$$
\frac{\LL^n([\mathbf{a},\mathbf{b}])}{(\max_\nu|a_\nu-b_\nu|)^n}\geq \al.
$$

The goal of this paper is to prove that the classes $\al\po AC$ coincide  for all  $\al\in(0,1)$. We show, however,   that the choice of $\de$ in \eqref{alac} cannot be made uniformly for all $\al\in(0,1)$.

Our method uses the class $1\po AC$ which was introduced in \cite{DRX}, see Definition~\ref{oneac} below.

\section{Preliminaries}
\label{prelres}

In 2002, Hencl \cite{Hencl2002} introduced the following class of absolutely continuous functions.

\begin{defn} \label{qac} (Hencl \cite{Hencl2002})
A function $f:\Om\to\R^l$ ($\Om\subset\R^n$ open)  is said to be in $\ACH$ (briefly $AC_H$) if there exists $\la\in(0,1)$  (equivalently, for all $\la\in(0,1)$) so that for all $\e>0$, there exists $\de>0$ so that for any finite collection of disjoint closed balls $\{B(\mathbf{x}_i,r_i) \subset\Om\}_{i=1}^k $  we have
$$\sum\limits_{i=1}^k \LL^n (B(\mathbf{x}_i,r_i))<\de\Rightarrow\sum\limits_{i=1}^k \osc^n(f,B(\mathbf{x}_i,\la r_i))<\e.$$
\end{defn}

Hencl  proved that  $AC_H\subset W^{1,n}_{loc}$ and that all functions in $AC_H$ are  differentiable a.e. and satisfy the Luzin (N) property and  the change of variables formula.

Bongiorno \cite{Bongiorno2009} introduced a modification of the notion of $\al$-absolute continuity in the spirit of Definition~\ref{qac}. To define it, we will use the following notation. 

Given interval $[\mathbf{x},\mathbf{y}]\subset \R^n$, we denote $|f([\mathbf{x},\mathbf{y}])|=|f(\mathbf{y})-f(\mathbf{x})|$, and for $0<\la<1$, we denote by $~^\la[\mathbf{x},\mathbf{y}]$ the interval with center $(\mathbf{x}+\mathbf{y})/2$ and sides of length $\la(y_\nu-x_\nu)$, $\nu=1,\dots,n$.

\begin{defn} \label{defalphahac} (Bongiorno \cite{Bongiorno2009})
 A function $f:\Om\to\R^l$ ($\Om\subset\R^n$ open)  is said to be in $\al\po\ACH$ (briefly $\al\po AC_H$) if there exists $\la\in(0,1)$ (equivalently, for all $\la\in(0,1)$) so that for all $\e>0$, there exists $\de>0$ such that for any finite collection of disjoint $\al$-regular intervals 
$\{ [\mathbf{a}_i,\mathbf{b}_i]\subset\Om \}_{i=1}^k$ 
we have 
\begin{equation}\lb{alhac}
\sum_{i=1}^k \LL^n ([\mathbf{a}_i,\mathbf{b}_i])<\de \Rightarrow\sum_{i=1}^k|f(^\la[\mathbf{a}_i,\mathbf{b}_i)])|^n<\e.
\end{equation}
\end{defn}

 Bongiorno \cite{Bongiorno2009}  proved that for all $0<\al<1$, the classes  $\al\po AC_H$ and $AC_H$ coincide.

\begin{thm}\label{bonach} ( \cite[Theorem~3]{Bongiorno2009})
For every $0<\al<1$,  $$\al\po\ACH= \ACH.$$
\end{thm}

\section{The main result}
\label{secalac}

In \cite{DRX} we introduced  an analog of Bongiorno's notion for $\al=1$, which will be an important tool for the main result of this paper.

\begin{defn} \label{oneac}
We say that a function $f:\Om \rightarrow \R^l$ ($\Om \subset \R^n$ open) is {\it $1$-absolutely continuous}, denoted $1\po\AC$  or $1\po\ACq$,  if  for every $ \e > 0$, there exists $\delta>0$, such that for any finite collection of disjoint $1$-regular intervals $\{[\mathbf{a}_i,\mathbf{b}_i] \subset\Om\}_{i=1}^k $ we have $$\sum_{i=1}^k \LL^n ([\mathbf{a}_i,\mathbf{b}_i])<\de\Rightarrow \sum_{i=1}^k|f(\mathbf{a}_i)-f(\mathbf{b}_i)|^n<\e.$$
\end{defn}

Note that  for all $\al<1$, every $\al$-regular  interval is also 1-regular. Thus  for all $\al<1$,
$$ \al\po AC\subseteq1\po AC.$$

Properties of the class $1\po AC$ were studied in \cite{DRX}, where it was shown that it  contains many pathological functions. In particular functions in $1\po AC$  don't  need to be continuous and even if they are differentiable a.e. they do not need to belong to the Sobolev space $W^{1,n}_{loc}$. However,  the class $1\po AC$ is very useful for characterizing the classes $\al\po AC$.


\begin{thm} \label{1capH}
For all $0<\al<1$ we have $$\al\po\AC=1\po\AC\cap \ACH.$$
\end{thm}

\begin{proof}
By definitions, $\al\po\AC \subset 1\po\AC$, and by  \cite[Theorem~7]{Bongiorno2005},  $\al\po\AC \subset \ACH$, which proves that  
$$\al\po\AC \subset 1\po\AC\cap \ACH.$$

For the other direction, let $f\in 1\po\AC\cap \ACH$ and $\e>0$. Let $\beta=(\frac{\al}{2-\al})^n$ and $\la=1-\frac{\al}{2}$. By Theorem~\ref{bonach}, $f\in\beta\po\ACH$ and  there exists $\de_1>0$ so that for each  finite family of disjoint $\beta$-regular intervals $\{[\mathbf{a}_i,\mathbf{b}_i]\subset \Om\}$ 
$$\sum_{i}^{} \LL^n([\mathbf{a}_i,\mathbf{b}_i])<\de_1\Rightarrow\sum_{i}^{}|f(^{\la}[\mathbf{a}_i,\mathbf{b}_i])|^n<\frac{\e}{3^{n+1}}.$$

Since $f\in 1\po\AC$, there exists $\de_2>0$ so that for each  finite family of disjoint 1-regular intervals $\{[\mathbf{a}_i,\mathbf{b}_i]\subset \Om\}$ 
$$\sum_{i}^{} \LL^n([\mathbf{a}_i,\mathbf{b}_i])<\de_2\Rightarrow\sum_{i}^{}|f([\mathbf{a}_i,\mathbf{b}_i])|^n<\frac{\e}{3^{n+1}}.$$

Let $\de=\min\{\de_1,\de_2\}$ and  $\{[\mathbf{a}_i,\mathbf{b}_i]\subset \Om\}$ be  a finite family of disjoint $\al$-regular intervals with $\sum_{i} \LL^n([\mathbf{a}_i,\mathbf{b}_i])<\de$.  We will show that $\sum_{i}|f([\mathbf{a}_i,\mathbf{b}_i])|^n<\e$.

For each $[\mathbf{a}_i,\mathbf{b}_i]$ from this family, let $\eta_i=\max\limits_{1\leq\nu\leq n}(b_{i\nu}-a_{i\nu})$.
Since  $[\mathbf{a}_i,\mathbf{b}_i]$ is an $\al$-regular interval, we get that
$$\prod_{j=1}^n\left(\frac{b_{ij}-a_{ij}}{\eta_i}\right)\ge \al.$$
Thus, for every $j=1,\dots,n$,
$$\frac{b_{ij}-a_{ij}}{\eta_i}\ge \al.$$

Let
$$\mathbf{c}_i=\left(a_{ij}+\frac{\al}{4}\eta_i\right)_{j=1}^n,\ \ \mathbf{d}_i=\left(b_{ij}-\frac{\al}{4}\eta_i\right)_{j=1}^n.$$
Then,  for every $j=1,\dots,n$, $c_{ij}<d_{ij}$,  and for every $i$
$$\max_j(d_{ij}-c_{ij})=\max_j\left(b_{ij}-a_{ij}-\frac\al2 \eta_i\right)=\left(1- \frac\al2\right) \eta_i.$$
Hence
\begin{equation*}
\begin{split}
\frac{\LL^n([\mathbf{c}_i,\mathbf{d}_i])}{\max_j(d_{ij}-c_{ij})^n}&=\prod_{j=1}^n\left(\frac{b_{ij}-a_{ij}-\frac\al2 \eta_i}{\left(1- \frac\al2\right)\eta_i}\right)\\
&\ge \prod_{j=1}^n\left(\frac{\al\eta_i-\frac\al2 \eta_i}{\left(1- \frac\al2\right)\eta_i}\right)=\left(\frac{\al}{2-\al}\right)^n=\be.
\end{split}
\end{equation*}

Thus interval $[\mathbf{c}_i,\mathbf{d}_i]$ is $\be$-regular.

Let
\begin{align*}
\overline{\mathbf{c}_i}&=\Big(\frac{a_{ij}+b_{ij}}{2}-\frac{1}{\la}\Big(\frac{a_{ij}+b_{ij}}{2}-a_{ij}-\frac{\al}{4}\eta_i\Big)\Big)_{j=1}^n\\&=\Big(\frac{c_{ij}+d_{ij}}{2}-\frac{1}{\la}\Big(\frac{d_{ij}-c_{ij}}{2}\Big)\Big)_{j=1}^n, \\
\overline{\mathbf{d}_i}&=\Big(\frac{a_{ij}+b_{ij}}{2}+\frac{1}{\la}\Big(b_{ij}-\frac{\al}{4}\eta_i-\frac{a_{ij}+b_{ij}}{2}\Big)\Big)_{j=1}^n\\
&=\Big(\frac{c_{ij}+d_{ij}}{2}+\frac{1}{\la}\Big(\frac{d_{ij}-c_{ij}}{2}\Big)\Big)_{j=1}^n.
\end{align*}

Then  $[\overline{\mathbf{c}_i},\overline{\mathbf{d}_i}]\subset [\mathbf{a}_i,\mathbf{b}_i]$,  and 
$$^{\la}[\overline{\mathbf{c}_i},\overline{\mathbf{d}_i}]=[\mathbf{c}_i,\mathbf{d}_i],$$
$$|f(\mathbf{a}_i,\mathbf{b}_i)|^n\leq 3^n\left[|f(\mathbf{a}_i,\mathbf{c}_i)|^n+|f(\mathbf{c}_i,\mathbf{d}_i)|^n+|f(\mathbf{d}_i,\mathbf{b}_i)|^n\right],$$
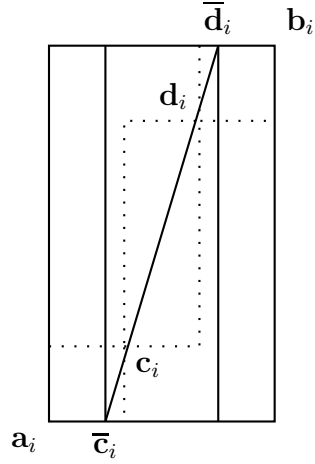
$$\sum_{i}^{}(\LL^n[\mathbf{a}_i,\mathbf{c}_i]+\LL^n[\mathbf{d}_i,\mathbf{b}_i])<\sum\limits_{i}^{}\LL^n[\mathbf{a}_i,\mathbf{b}_i]<\de,$$
$$\sum\limits_{i}\LL^n[\overline{\mathbf{c}_i},\overline{\mathbf{d}_i}]<\sum\limits_{i}^{}\LL^n[\mathbf{a}_i,\mathbf{b}_i]<\de.
$$

\begin{figure}
\begin{tikzpicture}[thick]
\draw (0,0) -- (0,5)-- (3,5) -- (3,0) -- cycle;
\coordinate [label=-135:$\mathbf{a}_i$] (a) at (0,0);
\coordinate [label=45:$\mathbf{b}_i$] (b) at (3,5);
\draw[loosely dotted] (0,1) -- (1,1)-- (1,0);
\draw[loosely dotted] (2,5) -- (2,4)-- (3,4);
\draw[loosely dotted] (1,1) -- (1,4)-- (2,4) --(2,1) -- cycle;
\coordinate [label=-45:$\mathbf{c}_i$] (c) at (1,1);
\coordinate [label=135:$\mathbf{d}_i$] (d) at (2,4);
\coordinate [label=-90:$\overline{\mathbf{c}}_i$] (c') at (0.75,0);
\coordinate [label=90:$\overline{\mathbf{d}}_i$] (d') at (2.25,5);
\draw (c')--(d');
\draw (c')--(0.75,5);
\draw (d')--(2.25,0);
\end{tikzpicture}\\
\end{figure}

Since the intervals $[\mathbf{a}_i,\mathbf{c}_i]$ and $[\mathbf{d}_i,\mathbf{b}_i]$ are  1-regular, and the intervals $[\overline{\mathbf{c}_{i}},\overline{\mathbf{d}_{i}}]$ are $\beta$-regular we get
\begin{align*}
\sum_{i}&|f(\mathbf{a}_i,\mathbf{b}_i)|^n \leq \\
&3^n\Big[\sum_{i}^{}|f([\mathbf{a}_i,\mathbf{c}_i])|^n+\sum_{i}^{}|f(^{\la}[\overline{\mathbf{c}_i},\overline{\mathbf{d}_i}])|^n+
 \sum\limits_{i}^{}|f([\mathbf{d}_i,\mathbf{b}_i])|^n\Big]\\
 &<\frac{\e}{3}+\frac{\e}{3}+\frac{\e}{3} = \e.
\end{align*}
\end{proof} 

As an immediate consequence of Theorem~\ref{1capH} we obtain the following corollary.

\begin{cor} \label{alphaeq}
For all $0<\al,\be<1$ and all $n,l\in \N$ we have $$\al\po\AC=\be\po\AC.$$

However the choice of $\de>0$ in \eqref{alac} cannot be made uniformly for all $\al\in(0,1)$.
\end{cor}

\begin{proof} To prove the final statement, 
suppose that $f$ is a function so that  for all $ \varepsilon>0$, 
 there exists $\de>0$, such that for all $\al\in(0,1)$ and for any finite collection of disjoint $\al$-regular intervals 
$\{ [\mathbf{a}_i,\mathbf{b}_i]\subset\Om \}_{i=1}^k$,  the implication
\eqref{alac} holds. Note that every nontrivial interval in $\R^n$, i.e. an interval $[\mathbf{x},\mathbf{y}]$  such that $x_\nu <y_\nu$ for all $\nu=1,\dots,n$, is $\al$-regular for some $\al\in(0,1)$. Thus \eqref{alac} holds for any finite collection of nontrivial intervals. By \cite[Theorem~3.1]{DRX} this implies that $f$ is constant.
\end{proof}

\begin{rem} \lb{open}
In \cite{B2009top} Bongiorno  introduced another generalization of absolute continuity, a class $AC^n_\La(\Om,\R^m)$, briefly $AC_\La$, which strictly contains the class $AC_H$, and such that  all functions in $AC_\La$ are differentiable a.e. and satisfy the Luzin (N) property, but $AC^n_\La(\Om,\R^l)\not\subset W^{1,n}_{loc}(\Om,\R^l)$. 
It would be interesting to determine 
what are the  classes  $1\po AC \cap AC_\La$ and $1\po AC \cap AC_\La\cap W^{1,n}_{loc}$ (since, by \cite{DRX}, $1\po AC \not\subset W^{1,n}_{loc}$).
\end{rem}


\begin{thebibliography}{10}

\bibitem{AAK11}
{\sc D.~Apatsidis, S.~A. Argyros, and V.~Kanellopoulos}, {\em Hausdorff
  measures and functions of bounded quadratic variation}, Trans. Amer. Math.
  Soc., 363 (2011), pp.~4225--4262.

\bibitem{Bongiorno2005}
{\sc D.~Bongiorno}, {\em Absolutely continuous functions in {$\Bbb R^n$}}, J.
  Math. Anal. Appl., 303 (2005), pp.~119--134.

\bibitem{Bongiorno2009}
{\sc D.~Bongiorno}, {\em On the {H}encl's
  notion of absolute continuity}, J. Math. Anal. Appl., 350 (2009),
  pp.~562--567.

\bibitem{B2009top}
{\sc D.~Bongiorno},  {\em On the problem of
  regularity in the {S}obolev space {$W_{\rm loc}^{1,n}$}}, Topology Appl., 156
  (2009), pp.~2986--2995.

\bibitem{Cal51}
{\sc A.~P. Calder{\'o}n}, {\em On the differentiability of absolutely
  continuous functions}, Rivista Mat. Univ. Parma, 2 (1951), pp.~203--213.

\bibitem{C41}
{\sc L.~Cesari}, {\em Sulle funzioni assolutamente continue in due variabili},
  Ann. Scuola Norm. Super. Pisa (2), 10 (1941), pp.~91--101.

\bibitem{CA1933}
{\sc J.~A. Clarkson and C.~R. Adams}, {\em On definitions of bounded variation
  for functions of two variables}, Trans. Amer. Math. Soc., 35 (1933),
  pp.~824--854.

\bibitem{DRX}
{\sc M. Dymond, B.~Randrianantoanina and H.~Xu}, {\em On interval based generalizations of absolute continuity for functions on $\R^n$},
\newblock preprint.

\bibitem{Hencl2002}
{\sc S.~Hencl}, {\em On the notions of absolute continuity for functions of
  several variables}, Fund. Math., 173 (2002), pp.~175--189.

\bibitem{Hobson}
{\sc E.~W. Hobson}, {\em The theory of functions of a real variable and the
  theory of {F}ourier's series. {V}ol. {I}}, Dover Publications Inc., New York,
  N.Y., 1958.

\bibitem{Maly1999}
{\sc J.~Mal{\'y}}, {\em Absolutely continuous functions of several variables},
  J. Math. Anal. Appl., 231 (1999), pp.~492--508.

\bibitem{RR}
{\sc T.~Rado and P.~V. Reichelderfer}, {\em Continuous transformations in
  analysis. {W}ith an introduction to algebraic topology}, Springer-Verlag,
  Berlin, 1955.

\end{thebibliography}

\def\cprime{$'$}

\end{document}